\documentclass[11pt,twoside]{amsart}
\usepackage[russian,english]{babel}
\Russian

\usepackage{amssymb}
\usepackage{latexsym}

\usepackage{color}

\catcode`\@=11
\def\omathop#1#2#3{\let\temp=#1\def\letter{#2}
  \ifcat#3_ \let\next\@@olim\else\let\next\@olim\fi\next#3}
\def\@olim{\letter\text{-}\!\temp}
\def\@@olim_#1{\mathchoice{
   \setbox0=\hbox{$\displaystyle\letter\text{-}\!\temp\!\text{-}\letter$}
   \setbox2=\hbox{$\displaystyle\temp$}
   \setbox4=\hbox{$\scriptstyle#1$}
   \dimen@=\wd4 \advance\dimen@ by -\wd2 \divide\dimen@ by2
   \def\next{\letter\text{-}\!\temp_{\hbox to 0pt{\hss$\scriptstyle#1$\hss}}
     \hskip\dimen@}
   \ifdim\wd2>\wd4 \def\next{\@olim_{#1}}\fi
   \ifdim\wd4>\wd0 \def\next{\mathop{\llap{$\letter$-}\!\temp}\limits_{#1}}\fi
   \next}
   {\@olim_{#1}}{\@olim_{#1}}{\@olim_{#1}}}

\def\bolim{\omathop{\lim}{bo}}

\catcode`\@=13

\newcommand{\reduce}{\mskip-2mu}
\newcommand{\ls}{\reduce\left\bracevert\reduce\vphantom{X}}
\newcommand{\rs}{\reduce\vphantom{X}\reduce\right\bracevert\reduce}

\theoremstyle{plain}
\newtheorem{thm}{Theorem}[section]

\newtheorem{prop}[thm]{Proposition}

\newtheorem{lemma}[thm]{Lemma}
\newtheorem{rem}[thm]{Remark}

\newtheorem{example}[thm]{Example}

\theoremstyle{definition}
\newtheorem{definition}[thm]{Definition}

\numberwithin{equation}{section}

\begin{document}

\title{Narrow orthogonally additive operators on lattice-normed spaces}

\author{Xiaochun Fang}

\address{Department of Mathematics, Tongji University \\
Shanghai 200092, CHINA}

\email{xfang\copyright tongji.edu.cn}

\author{Marat Pliev}

\address{South Mathematical Institute of the Russian Academy of Sciences\\
Vladikavkaz, 362027 Russia}

\email{maratpliev\copyright gmail.com}

%\centerline{\today}



\begin{abstract}
The aim of this article is to extend results of M.~Popov and second named author  about orthogonally additive narrow operators on vector lattices. The main object of our investigations are an orthogonally additive narrow operators between  lattice-normed spaces. We prove that every $C$-compact laterally-to-norm continuous orthogonally additive  operator from a Banach-Kantorovich space $V$ to a Banach lattice $Y$ is narrow. We also show that every dominated Uryson operator from Banach-Kantorovich space over an atomless Dedekind complete vector lattice $E$ to a sequence Banach lattice $\ell_p(\Gamma)$ or $c_0(\Gamma)$ is narrow. Finally,   we  prove   that if an  orthogonally additive dominated operator $T$ from lattice-normed space $(V,E)$ to Banach-Kantorovich space $(W,F)$ is order narrow then the order narrow   is its  exact dominant $\ls T\rs$.
\end{abstract}

\maketitle

%\tableofcontents

\section{Introduction}

Narrow operators in framework of the theory of ordered spaces were introduces in \cite{MMP}. Today the theory of linear narrow operators is a very active area of Functional Analysis (see \cite{PR}). Firstly, this class of  operators in the framework of the theory of vector lattices was introduced in \cite{MMP}. In \cite{PP} Pliev and Popov have considered a wide class of narrow nonlinear maps which called orthogonally additive operators. This class of operators acting between vector lattices was introduced and studied in 1990 by Maz\'{o}n and Segura de Le\'{o}n \cite{Maz-1,Maz-2}, and then extended to lattice-normed spaces by Kusraev and the second named author \cite{Ku-1,Ku-2,Pl-3}. In the present paper we generalize the main results of \cite{PP} on orthogonally additive  narrow operators on vector lattices to a wider class which includes nonlinear  orthogonally additive operators between vector-valued function spaces. We consider orthogonally additive narrow operators in the framework of lattice-normed spaces. The notion of a lattice-normed space was firstly introduced by Kantorovich in the first part of 20th century \cite{Ka}. Later, Kusraev and his school had provided a deep theory. Orthogonally additive operators on lattice-normed spaces were investigated by Pliev and collaborates in \cite{Ben, Get, Ku-1,Ku-2, Pl-3,Pl-4,Pl-5,PP-1}. \footnote{The    second  named author was  supported by the  Russian Foundation of Basic Research, the grant number 15-51-53119}

\section{Preliminaries}

The  goal of this section is to introduce some basic definitions and facts. General information on vector lattices,
Banach spaces and lattice-normed spaces the reader can find in the books \cite{Al,Ku,Li-1,Li-2,LZ}.

Consider a vector space $V$ and a real  archimedean vector lattice
$E$. A map $\ls \cdot\rs:V\rightarrow E$ is a \textit{vector norm} if it satisfies the following axioms:
\begin{enumerate}
  \item[1)] $\ls v \rs\geq 0;$\,\, $\ls v\rs=0\Leftrightarrow v=0$;\,\,$(\forall v\in V)$.
  \item[2)] $\ls v_1+v_2 \rs\leq \ls v_1\rs+\ls v_2 \rs;\,\, ( v_1,v_2\in V)$.
  \item[3)] $\ls\lambda  v\rs=|\lambda|\ls v\rs;\,\, (\lambda\in\Bbb{R},\,v\in V)$.
\end{enumerate}
A vector norm is called \textit{decomposable} if
\begin{enumerate}
  \item[4)] for all $e_{1},e_{2}\in E_{+}$ and $x\in V$ from $\ls x\rs=e_{1}+e_{2}$ it follows that there exist $x_{1},x_{2}\in V$ such that $x=x_{1}+x_{2}$ and $\ls x_{k}\rs=e_{k}$, $(k:=1,2)$.
\end{enumerate}

A triple $(V,\ls\cdot\rs,E)$ (in brief $(V,E),(V,\ls\cdot\rs)$ or $V$ with default parameters omitted) is a \textit{lattice-normed space} if $\ls\cdot\rs$ is a $E$-valued vector norm in the vector space $V$. If the norm $\ls\cdot\rs$ is decomposable then the space $V$ itself is called decomposable. We say that a net $(v_{\alpha})_{\alpha\in\Delta}$ {\it $(bo)$-converges} to an element $v\in V$ and write $v=\bolim v_{\alpha}$ if there exists a decreasing net $(e_{\gamma})_{\gamma\in\Gamma}$ in $E_{+}$ such that $\inf_{\gamma\in\Gamma}(e_{\gamma})=0$ and for every $\gamma\in\Gamma$ there is an index $\alpha(\gamma)\in\Delta$ such that $\ls v-v_{\alpha(\gamma)}\rs\leq e_{\gamma}$ for all $\alpha\geq\alpha(\gamma)$. A net $(v_{\alpha})_{\alpha\in\Delta}$ is called \textit{$(bo)$-fundamental} if the net $(v_{\alpha}-v_{\beta})_{(\alpha,\beta)\in\Delta\times\Delta}$ $(bo)$-converges to zero. A lattice-normed space is called {\it $(bo)$-complete} if every $(bo)$-fundamental net $(bo)$-converges to an element of this space. Let $e$ be a positive element of a vector lattice $E$.
By $[0,e]$ we denote the set $\{v\in V:\,\ls v\rs\leq e\}$.
A set $M\subset V$ is called  $\text{(bo)}$-{\it bounded } if there exists
$e\in E_{+}$ such that  $M\subset[0,e]$. Every decomposable $(bo)$-complete lattice-normed space is called a {\it Banach-Kantorovich space} (a BKS for short).

Let $(V,E)$ be a lattice-normed space.  A subspace $V_{0}$ of $V$ is called a $\text{(bo)}$-ideal of $V$ if for $v\in V$ and $u\in V_{0}$, from $\ls v\rs\leq\ls u\rs$ it follows that $v\in V_{0}$. A subspace $V_{0}$ of a decomposable lattice-normed space $V$  is a $\text{(bo)}$-ideal if and only if $V_{0}=\{v\in V:\,\ls v\rs\in L\}$, where $L$ is an order ideal in $E$ (\cite{Ku}, Prop.~2.1.6.1). Let $V$ be a lattice-normed space and $y,x\in V$. If $\ls x\rs\bot\ls y\rs=0$ then we call the elements $x,y$ {\it disjoint} and write $x\bot y$. The equality $x=\coprod_{i=1}^{n}x_{i}$ means that $x=\sum\limits_{i=1}^{n}x_{i}$ and $x_{i}\bot x_{j}$ if $i\neq j$. An element $z\in V$ is called a {\it component} or a \textit{fragment} of $x\in V$ if $0\leq \ls z\rs\leq\ls x\rs$ and $x\bot(x-z)$. Two fragments $x_{1},x_{2}$ of $x$  are called \textit{mutually complemented} or $MC$, in short, if $x=x_1+x_{2}$. The notations $z\sqsubseteq x$ means that $z$ is a fragment of $x$. According to (\cite{Al}, p.111) an element $e>0$ of a vector lattice $E$ is called an {\it atom}, whenever $0\leq f_{1}\leq e$, $0\leq f_{2}\leq e$ and $f_{1}\bot f_{2}$ imply that either $f_{1}=0$ or $f_{2}=0$. A vector lattice $E$ is  atomless if there is no atom $e\in E$.

The following object will be often used in different constructions below. Let $V$ be a lattice-normed space and $x\in V$. A sequence $(x_{n})_{n=1}^{\infty}$ is called a {\it disjoint tree} on $x$ if $x_{1}=x$ and $x_{n}=x_{2n}\coprod x_{2n+1}$ for each $n\in\Bbb{N}$. It is clear that all $x_{n}$ are fragments of $x$. All lattice-normed spaces below we consider to be decomposable.

Consider some important examples of lattice-normed spaces. We begin with simple extreme cases, namely vector lattices and normed spaces. If $V=E$ then the modules of an element can be taken as its lattice norm: $\ls v\rs:=|v|=v\vee(-v);\,v\in E$. Decomposability of this norm easily follows from the Riesz Decomposition Property holding in every vector lattice. If $E=\Bbb{R}$ then $V$ is a normed space.

Let $Q$ be a compact  and let $X$ be a Banach space. Let $V:=C(Q,X)$ be the space of  continuous vector-valued functions from $Q$ to $X$. Assign $E:=C(Q,\Bbb{R})$. Given $f\in V$, we define its lattice norm by the relation $\ls f\rs:t\mapsto\|f(t)\|_{X}\,(t\in Q)$. Then $\ls\cdot\rs$ is a decomposable norm (\cite{Ku}, Lemma~2.3.2).

Let $(\Omega,\Sigma,\mu)$ be a $\sigma$-finite measure space,
let $E$ be an order-dense ideal in  $L_{0}(\Omega)$ and let $X$ be a Banach space.
By $L_{0}(\Omega,X)$ we denote the space of (equivalence classes ) of  Bochner $\mu$-measurable vector functions
acting from $\Omega$ to $X$. As usual, vector-functions are equivalent if they have
equal values at almost all points of the set $\Omega$. If $\widetilde{f}$ is the coset of a measurable vector-function $f:\Omega\rightarrow X$ then $t\mapsto\|f(t)\|$,$(t\in\Omega)$ is a scalar measurable function whose coset is denoted by the symbol $\ls\widetilde{f}\rs\in L_{0}(\mu)$. Assign by definition
$$
E(X):=\{f\in L_{0}(\mu,X):\,\ls f\rs\in E\}.
$$
Then $(E(X),E)$ is a lattice-normed space with a decomposable norm (\cite{Ku}, Lemma~2.3.7).
If $E$ is a Banach lattice then the lattice-normed space $E(X)$ is  a Banach space with respect to the norm $|\|f|\|:=\|\|f(\cdot)\|_{X}\|_{E}$.

Let $E$ be a Banach lattice and let $(V,E)$ be a lattice-normed space. By definition,
$\ls x\rs\in E_{+}$ for every $x\in V$, and we can introduce some \textit{mixed norm} in $V$ by the formula
$$
\||x|\|:=\|\ls x\rs\|\,\,\,(\forall\, x\in V).
$$
The normed space $(V,\||\cdot|\|)$ is called a \textit{space with a mixed norm}.
In view of the inequality $|\ls x\rs-\ls y\rs|\leq\ls x-y\rs$ and monotonicity of the norm in $E$, we have
$$
\|\ls x\rs-\ls y\rs\|\leq\||x-y|\|\,\,\,(\forall\, x,y\in V),
$$
so a vector norm is a norm continuous operator from $(V,\||\cdot|\|)$ to $E$. A lattice-normed space $(V,E)$ is called
a \textit{Banach space with a mixed norm} if the normed space $(V,\||\cdot|\|)$ is complete with respect to the norm convergence.

\begin{definition} \label{def:ddmjf0}
Let $E$ be a vector lattice, and let $F$ be a real linear space. An operator $T:E\rightarrow F$ is called \textit{orthogonally additive} if $T(x+y)=T(x)+T(y)$ whenever $x,y\in E$ are disjoint.
\end{definition}

It follows from the definition that $T(0)=0$. It is immediate that the set of all orthogonally additive operators is a real vector space with respect to the natural linear operations.

\begin{definition}
Let $E$ and $F$ be vector lattices. An orthogonally additive operator $T:E\rightarrow F$ is called:
\begin{itemize}
  \item \textit{positive} if $Tx \geq 0$ holds in $F$ for all $x \in E$;
  \item \textit{order bounded} if $T$ maps order bounded sets in $E$ to order bounded sets in $F$.
\end{itemize}
An orthogonally additive, order bounded operator $T:E\rightarrow F$ is called an \textit{abstract Uryson} operator.
\end{definition}
For example, any linear  operator $T\in L_{+}(E,F)$ defines a positive abstract Uryson operator by $G (f) = T |f|$ for each $f \in E$.

The set of all abstract Uryson operators from $E$ to $F$ we denote by $\mathcal{U}(E,F)$. Consider some examples.
The most famous one is the nonlinear integral Uryson operator.

Consider the following order in $\mathcal{U}(E,F):S\leq T$ whenever $T-S$ is a positive operator. Then $\mathcal{U}(E,F)$
becomes an ordered vector space.  If a vector lattice $F$ is Dedekind complete we have the following theorem.
\begin{thm}(\cite{Maz-1},Theorem~3.2)\label{th-1}.
Let $E$ and $F$ be a vector lattices, $F$ Dedekind complete. Then $\mathcal{U}(E,F)$ is a Dedekind complete vector lattice. Moreover for $S,T\in \mathcal{U}(E,F)$ and for $f\in E$ following hold
\begin{enumerate}
\item~$(T\vee S)(f):=\sup\{Tg+Sh:\,f=g\sqcup h\}$.
\item~$(T\wedge S)(f):=\inf\{Tg+Sh:\,f=g\sqcup h\}.$
\item~$(T)^{+}(f):=\sup\{Tg:\,g\sqsubseteq f\}$.
\item~$(T)^{-}(f):=-\inf\{Tg:\,g;\,\,g\sqsubseteq f\}$.
\item~$|Tf|\leq|T|(f)$.
\end{enumerate}
\end{thm}

\begin{definition} \label{def:even}
Let $E$ be a vector lattice  and $X$ a vector space. An orthogonally additive map $T:E\to X$ is called  even if $T(x)=T(-x)$ for every $x\in E$. If $E,F$ are vector lattices, the set of all even abstract Uryson operators from $E$ to $F$ we denote by $\mathcal{U}^{ev}(E,F)$.
\end{definition}
If $E,F$ are vector lattices with $F$ Dedekind complete, the space $\mathcal{U}^{ev}(E,F)$ is not empty. Indeed, for every $T\in\mathcal{U}(E,F)$ by (\cite{Maz-1},Proposition~3.4) there exists an even operator $\widetilde{T}\in U_{+}^{ev}(E,F)$ which is defined by the formula,
$$
\widetilde{T}f=\sup\{|T|g:\,|g|\leq|f|\}.
$$

\begin{lemma}(\cite{PP-1}, Lemma~3.2.)\label{lemma-ev}
Let $E,F$ be vector lattices with $F$ Dedekind complete. Then $\mathcal{U}^{ev}(E,F)$ is a Dedekind complete sublattice of $\mathcal{U}(E,F)$.
\end{lemma}

\begin{definition} \label{def:dom}
Let $(V,E)$ and $(W,F)$ be lattice-normed spaces. A map $T:V\to W$ is called {\it orthogonally additive} if $T(u+v)=Tu+Tv$ for every $u,v\in V,\,u\bot v$. An orthogonally additive map $T:V\to W$ is called a {\it dominated Uryson} operator if there exists $S\in\mathcal{U}_{+}^{ev}(E,F)$ such that $\ls Tv\rs\leq S\ls v\rs$ for every $v\in V$. In this case we say that $S$ is a \textit{dominant} for $T$. The set of all dominants of the operator $T$ is denoted by $\text{Domin}(T)$. If there is the least element in $\text{Domin}(T)$ with respect to the order induced by $\mathcal{U}_{+}^{ev}(E,F)$ then it is called the {\it least} or the {\it exact dominant} of $T$ and is denoted by
$\ls T\rs$. The set of all dominated Uryson operators from $V$ to $W$ is denoted by $\mathcal{D}_{U}(V,W)$.
\end{definition}

\begin{example} \label{Ex0}
Let $X,Y$ be normed spaces.  Consider the lattice-normed spaces $(X,\Bbb{R})$ and $(Y,\Bbb{R})$. Then a given map $T:X\to Y$ is an element of $\mathcal{D}_{U}(X,Y)$ if and only if there exists  an even function $f:\Bbb{R}\to\Bbb{R}_{+}$ such that $f(0)=0$, the set $f(D)$ is bounded for an every bounded subset $D\subset\Bbb{R}$ and the inequality  $\|Tx\|\leq f(\|x\|)$ holds for every $x\in X$.
\end{example}

\begin{example} \label{Ex1}
Let $E,F$ be vector lattices with $F$ Dedekind complete.  Consider the lattice-normed spaces $(E,E)$ and $(F,F)$ where the lattice valued  norms coincide with the modules. We may show that the vector space $\mathcal{D}_{U}(E,F)$ coincide with $\mathcal{U}(E,F)$. Indeed, if $T\in\mathcal{D}_{U}(E,F)$, then there exists $S\in\mathcal{U}_{+}^{ev}(E,F)$ such that
$|Tx|\leq S|x|$ for every $x\in E$. Thus, $T$ is order bounded. If $T\in\mathcal{U}(E,F)$ then by (\cite{Maz-1}, Proposition~3.4) there exists $S\in\mathcal{U}_{+}^{ev}(E,F)$, so that $|Tf|\leq S(f)\leq S(|f|)$ and therefore $T\in\mathcal{D}_{U}(E,F)$.
\end{example}

\begin{example} \label{Ex2}
Let $(A,\Sigma,\mu)$  be a finite complete measure space, $E$ an order dense ideal in $L_{0}(\mu)$ and  $X$ a Banach space. Let $N:A\times X\to X$ be a function satisfying the following conditions:
\begin{enumerate}
  \item[$(C_{0})$] $N(t,0)=0$ for $\mu$-almost all $t\in A$;
  \item[$(C_{1})$] $N(\cdot,x)$ is Bochner $\mu$-measurable for all $x\in X$;
  \item[$(C_{2})$] $N(t,\cdot)$ is continuous with respect to the norm of $X$  for $\mu$-almost all $t\in A$.
  \item[$(C_{3})$] There exists a measurable function $M:A\times\Bbb{R}\to\Bbb{R}_{+}$,
  so that $M(t,\cdot)$ is increasing and $M(t,r)=M(t,-r)$ for $\mu$-almost all $t\in A$, $r\in\Bbb{R}$ and
   $$
   \sup\limits_{\|x\|\leq r}\|N(t,x)\|\leq M(t,r)\,\, \text{for $\mu$-almost all}\,\, t\in A, r\in\Bbb{R}.
   $$
\end{enumerate}
By $\text{Dom}(N)$ we denote the set of the Bochner $\mu$-measurable vector-function $f:A\to X$, so that $N(\cdot,f(\cdot))\in L_{1}(\mu,X)$. If $E(X)\subset\text{Dom}(N)$ and $M(\cdot,g(\cdot))\in L_{1}(\mu)$ for every $g\in E$, then is defined the orthogonally additive  operator $T:E(X)\to X$ by the formula
$$
Tf:=\int_{A}N(t,f(t))\,d\mu(t).
$$
Let us show that $T\in\mathcal{D}_{U}(E(X),X)$. Indeed
$$
\ls Tf\rs=\|Tf\|=\Big\|\int_{A}N(t,f(t))\,d\mu(t)\Big\|\leq\int_{A}\|N(t,f(t))\|\,d\mu(t)\leq
$$
$$
\leq\int_{A}M(t,\|f(t)\|)\,d\mu(t)=S\ls f\rs,
$$
where $S:E\to\Bbb{R}_{+}$ is the integral Uryson operator, $Se=\int_{A}M(t,e(t))\,d\mu(t)$ and $S$ is a dominant for $T$.
\end{example}

\section{Definition and some properties of orthogonally additive narrow operators}
\label{sec2}

In this section we introduce a new class of nonlinear operators in lattice-normed spaces and describe some of their properties.

\begin{definition} \label{def:nar1}
Let $(V,E)$ be a  lattice-normed space over an  atomless vector lattice $E$ and $X$ be a Banach space. A map $T: V \to X$ is called:
\begin{itemize}
  \item \emph{narrow }  if for every $v\in V$ and $\varepsilon > 0$ there exist $MC$ fragments $v_1$, $v_2$ of $v$ such that $\|Tv_1-Tv_2\|<\varepsilon$;
    \item \emph{strictly narrow } if for every $v\in V$ there exist $MC$ fragments $v_1$, $v_2$ of $v$ such that $Tv_1 = Tv_2  $.
  \end{itemize}
\end{definition}

Next is the corresponding new definition of an order narrow operator.

\begin{definition} \label{def:nar2}
Let $(V,E)$ and $(W,F)$ be  lattice-normed spaces  with $E$ atomless. A map $T:V\to W$ is called
\begin{itemize}
  \item \emph{order narrow }  if for every $v\in V$ there exists a net of decompositions $v = v_\alpha^{1} \sqcup v_\alpha^{2}$ such that $(Tv_{\alpha}^{1} - Tv_{\alpha}^{2} )\overset{(bo)}\longrightarrow 0$.
\end{itemize}
\end{definition}

The set of all orthogonally additive narrow operators from a lattice-normed space $(V,E)$ to a Banach space $X$ we denote by $\mathcal{NOA}(V,X)$.

\begin{lemma} \label{le:3}
Let $(V,E)$ be a lattice-normed space and let $(W,F)$ be a Banach space with a mixed norm. Then every $T\in\mathcal{NOA}(V,W)$ is order narrow.
\end{lemma}

\begin{proof}
Take an arbitrary  element $u\in V$. Let $\varepsilon_{n}:=\frac{1}{2^{n}}$  and let $u=u_{n}^{1}\sqcup u_{n}^{2}$, $n\in\Bbb{N}$ be a sequence of decomposition of the element $u$, such that
$$
\|\ls Tu_{n}^{1}-Tu_{n}^{2}\rs\|=|\|Tu_{n}^{1}-Tu_{n}^{2}|\|\leq\varepsilon_{n}.
$$
We set  $f_{n}=\sum\limits_{k=n}^{\infty}\ls Tu_{k}^{1}-Tu_{k}^{2}\rs$, $f_{n}\in F_{+}$, $f_{n}\downarrow 0$ and the following estimate holds $\ls Tu_{n}^{1}-Tu_{n}^{2}\rs\leq f_{n}$. Thus, $(Tu_{n}^{1}-Tu_{n}^{2})\overset{(bo)}\rightarrow 0$.
\end{proof}

The sets of orthogonally additive narrow and order narrow operators coincide if a vector lattice $F$ is good enough.

\begin{lemma}\label{le:4}
Let  $(V,E)$ and $(W,F)$ be the same as in Lemma~\ref{le:3} and let $F$ be a Banach lattice with order continuous norm. Then orthogonally additive operator  $T:V\rightarrow W$ is order narrow if and only if $T$ is narrow.
\end{lemma}

\begin{proof}
Let $T$ be an order narrow operator. Then for every  $u\in V$ there exist a net of decompositions
$u=u_{\alpha}^{1}\sqcup u_{\alpha}^{2}$, such that
$(Tu_{\alpha}^{1}- Tu_{\alpha}^{2})\overset{(bo)}\rightarrow 0$. Fix any $\varepsilon>0$. Using the fact that the norm in $F$ is order continuous we can find $\alpha_{0}\in\Lambda$ such that $\|\ls Tu_{\alpha}^{1}- Tu_{\alpha}^{2} \rs\|<\varepsilon$ for every $\alpha\geq\alpha_{0}$. In view of Lemma~\ref{le:3}, the converse is true.
\end{proof}

\section{$C$-compact operators and operators  to a sequence vector lattices}
\label{sec4}

In this section we investigate connections between narrow and $C$-compact orthogonally additive operators.
Firstly we give a definitions.

Recall that a net $(x_\alpha)$ in a lattice-normed space  $(V,E)$ \textit{laterally converges} to $x \in V$ if $x_\alpha \sqsubseteq x_\beta \sqsubseteq x$ for all indices $\alpha < \beta$ and $x_\alpha \overset{(bo)}\longrightarrow x$. In this case we write $x_\alpha \overset{\rm lat}\longrightarrow x$.

\begin{definition}
Let $(V,E)$ be a  lattice-normed space and  $F$ be a Banach space.
The orthogonally additive operator $T:V\to F$ is called
\begin{itemize}
\item \textit{laterally-to-norm} continuous provided $T$ sends laterally convergent nets in $(V,E)$ to norm convergent nets in $F$;
  \item \textit{generelalized AM}-compact (or $GAM$-compact for short), if $T(M)$ are precompact in $F$ for any $(bo)$-bounded set $M\subset V$;
  \item \textit{$C$-compact} if the sets $\{ Tg: v\in\mathcal{F}_{v}\}$ are precompact in $F$ for every $v\in V$.
\end{itemize}
The set of all $C$-compact dominated Uryson operators from $V$ to $F$ is denoted by $\mathcal{CD}(V,F)$.
\end{definition}

{\begin{example}
Let $(V,E)$ be a  lattice-normed space,  $F$ be a Banach space. Since $\mathcal{F}_{v}$ is $(bo)$-bounded by the element $\ls v\rs$ for any $v\in V$,  an every $GAM$-compact orthogonally additive operator $T:V\to F$ is the $C$-compact.
\end{example}

\begin{example}
Let $([0,1],\Sigma_{1},\mu)$ and $([0,1],\Sigma_{2},\nu)$ be two measure spaces, $E = C[0,1]$, which is a sublattice of $L_\infty(\mu)$, and $F = L_{\infty}(\nu)$. Consider the integral Uryson operator $T:E\to F$ with the kernel $K(s,t,r) = \mathbf{1}_{[0,1]}(t) \mathbf{1}_{[0,1]}(s) |r|$. Since the interval $[0,1]$
is connected, every numerical continuous function $f: [0,1] \to \mathbb R$ is an atom, that is, $f$ has no nonzero fragment and therefore $\{T(g): \, g \sqsubseteq f\}$ is a relatively compact set in $F$ for every $f\in E$. Take $u(t) = \mathbf{1}_{[0,1]}(t)$ and consider the order bounded set $D=\{f\in E:\,|f|\leq u\}$ in $E$. Then we have
$$
T(f)(s) = \int_0^1 \mathbf{1}_{[0,1]}(t) \mathbf{1}_{[0,1]}(s)|f(t)| \, d\mu(t) = \mathbf{1}_{[0,1]}(s) \int_0^1 |f(t)|\,d\mu(t).
$$
Observe that $T(D)$ is not relatively compact in $F$. Therefore $T$ is $C$-compact, but not $AM$-compact.
\end{example}

Now, we need the following known property of atoms in vector lattices.
\begin{prop}(\cite{LZ}, Theorem~26.4~(ii)) \label{pr}
For any two atoms $u,v$ in a vector lattice $E$, either $u \bot v$, or $v = \lambda u$ for some $0 \neq \lambda \in \mathbb R$.
\end{prop}
We say that a vector lattice $E$ is \emph{ discrete} if there is a collection $(u_i)_{i \in I}$ of atoms in $E^+$, called a \emph{generating collection of atoms}, such that $u_i \bot u_j$ for $i \neq j$ and for every $x \in E$ if $|x| \wedge u_i = 0$ for each $i \in I$ then $x = 0$.

Let $X$ be a Banach space, $E=\Bbb{R}^{n}$ for some $n\in\Bbb{N}$, and $E(X)$ be a K{\"{o}}the-Bochner  space of finite $X$-valued sequences. Observe that $E(X)$ is a lattice-normed space with the vector norm defined by the formula $\ls f\rs:=\|f(\cdot)\|_{X}$ for every $f\in E(X)$, and $E(X)$ is also a Banach space with a mixed norm $\|f\|_{E(X)}:=\|\ls f\rs\|_{E}$.
{\begin{example}
The identity operator $I:E(X)\to E(X)$ is the $C$-compact. Indeed, $E$ is finite-dimensional atomic Banach lattice, and therefore  the set $\mathcal{F}_{\ls f\rs}$ and consequently the set $\mathcal{F}_{f}$ has a finite number of the  nonzero fragments for every $f\in E(X)$. Hence the  $\mathcal{F}_{f}$ is the precompact set in $E(X)$.
\end{example}
Remark that a C-compact abstract Uryson operator $T:E \to F$ between Banach lattices $E$, $F$ with $F$  $\sigma$-Dedekind complete is AM-compact if, in addition, $T$ is uniformly continuous on order bounded subsets of $E$ \cite[Theorem~3.4]{Maz-2}.

Now we need a some auxiliary  lemmas. The following lemma is known as the lemma on rounding off coefficients (\cite{Kad-1}, p.~14).

\begin{lemma} \label{le:rounding}
Let  $(x_{i})_{i=1}^{n}$ be a finite collection of vectors in a finite dimensional normed space  $X$ and let $(\lambda_{i})_{i=1}^{n}$ be a collection of reals with $0\leq\lambda_{i}\leq 1$ for each $i$. Then there exists a collection $(\theta_{i})_{i=1}^{n}$ of numbers $\theta_{i}\in\{0,1\}$ such that
$$
\Big\|\sum\limits_{i=1}^{n}(\lambda_{i}-\theta_{i}) \, x_{i}\Big\|\leq\frac{\text{\rm dim} \, X}{2}\max_{i}\|x_{i}\|.
$$
\end{lemma}

\begin{lemma} \label{le:fkkkiryu8}
Let $(V,E)$ and $F$ be the same as in the Theorem~\ref{thm:ourmain} and $T: V \to F$ be an orthogonally additive laterally-to-norm continuous operator. If $e\in E_{+}$, $\ls v_{n}\rs\leq e$ and $v_{n}\bot v_{m}$ for each integers $n\neq m$ then
$\lim\limits_{n\to\infty} \|T (v_n)\| = 0$.
\end{lemma}

\begin{proof}
Since $V$ is Banach-Kantorovich space, the sequence $u_{n}=\sum\limits_{k=1}^{n}v_{k}$ laterally
converges to $u=\sum\limits_{k=1}^{\infty}v_{k}$
Then the laterally-to-norm continuity of $T$ implies that $Tu_{n}$
converges to $Tu$ in F. The sequence $(T(u_n))_{n=1}^\infty$ is fundamental, that is, $\lim\limits_{n,m \to \infty} \|T(u_n) - T(u_m)\| = 0$, hence
$$
\|T(u_n) - T(u_{n-1}) \| = \Bigl\|T\Big(\sum\limits_{k=1}^{n}v_{k}\Big)-T \Big(\sum\limits_{k=1}^{n-1}v_{k}\Big) \Bigr\|= \|T(u_n) \|
$$
implies $\lim\limits_{n\to\infty} \|T (u_n)\| = 0$.
\end{proof}

\begin{lemma} \label{le:4.5}
Let $(V,E)$ be a Banach-Kantorovich space over  atomless Dedekind complete vector lattice $E$, $F$ be a  finite dimensional Banach space, $T: V \to F$ be an orthogonally additive laterally-to-norm continuous operator and $v\in V$. Then there exist $MC$ fragments $v_{1},v_{2}$ of $v$ such that $\|T(v_1)\| = \|T(v_2) \|$.
\end{lemma}

\begin{proof}
Fix any $MC$ fragments $v_1, v_2$ of $v$. If $\|T(v_1)\| = \|T(v_2) \|$ then there is nothing to prove. With no loss of generality we may and do assume that $\|T(v_1)\| - \|T(v_2)\| > 0$. Consider the partially ordered set
$$
D = \{w \sqsubseteq v_1: \, \|T(v_1 - w)\|-\|T(v_{2}+w)\|\geq 0\}
$$
where $w_{1} \leq w_{2}$ if and only if $w_{1}\sqsubseteq w_{2}$. If $B \subseteq D$ is a chain then $w^{\star}=\vee B\in D$ by the laterally-to-norm continuity of $T$. By the Zorn lemma, there is a maximal element $w_{0}\in D$. Now we show that $\|T(v_{1}-w_{0})\|-\|T(v_{2}+w_{0})\|=0$. Suppose on the contrary that
$$
\alpha=\|T(v_{1}-w_{0})\|-\|T(v_{2}+w_{0})\|> 0.
$$
Since $E$ is atomless, we can choose  a fragment $0\neq f\sqsubseteq(v_{1}-w_{0})$ such that $\|T(f)\|<\frac{\alpha}{4}$ and $\|T(-f)\|<\frac{\alpha}{4}$ . Since $w_{0}\bot f$, $w_{0}+f\sqsubseteq v_{1}$  we have
\begin{align*}
&\|T(v_{1}-w_{0}-f)\| - \|T(v_{2}+w_{0}+f)\| = \\
&\|T(v_{1}-w_{0})+T(-f)\| - \|T(v_{2}+w_{0})+T(f)\| \geq \\
&\|T(v_{1}-w_{0})\| - \|T(-f)\| + \|T(v_2 + w_0)\| - \|T(f)\| > \frac{\alpha}{2} > 0,
\end{align*}
that contradicts the maximality of $w_{0}$.
\end{proof}

\begin{lemma} \label{le:iiyugte678}
Let $(V,E)$, $F$, $T:V\to F$ be the same as in the Lemma~\ref{le:4.5},  $v\in V$ and $(v_{n})_{n=1}^{\infty}$ be a disjoint tree on $v$. If $\|T(v_{2n})\| = \|T(v_{2n+1})\|$ for every $n\geq 1$ then
$$
\lim_{m \to \infty} \gamma_m = 0, \,\,\,\,\, \mbox{where} \,\,\,\,\, \gamma_m = \max\limits_{2^{m}\leq i<2^{m+1}} \|T(v_i)\|.
$$
\end{lemma}

\begin{proof}
Set $\varepsilon=\limsup_{m\to\infty}\gamma_{m}$ and prove that $\varepsilon = 0$, which will be enough for the proof. Suppose on the contrary that $\varepsilon>0$. Then for each $n\in\Bbb{N}$ we set
$$
\varepsilon_{n}=\limsup\limits_{m\to\infty}\max\limits_{2^{m}\leq i<2^{m+1},\,v_{i}\sqsubseteq v_{n}} \|T(v_i)\|.
$$
Hence, for each $m\in\Bbb{N}$ one has
$$
\max\limits_{2^{m}\leq i<2^{m+1}}\varepsilon_{i}=\varepsilon.\,\,\,\,\,(\star)
$$
Now we are going to construct a sequence of mutually disjoint elements $(v_{n_{j}})_{j=1}^{\infty}$ such that $\|T(v_{n_{j}})\| \geq \frac{\varepsilon}{2}$, that is impossible by Lemma~\ref{le:fkkkiryu8}. At the first step we choose $m_{1}$ so that $\max\limits_{2^{m_{1}}\leq i<2^{m_{1}+1}}\|T(v_i)\|\geq\frac{\varepsilon}{2}$. By $(\star)$, we choose $i_1$, $2^{m_{1}}\leq i_{1}<2^{m_{1}+1}$ so that $\varepsilon_{i_1}=\varepsilon$. Using $\|T(v_{2n})\| = \|T(v_{2n+1})\|$, we choose $n_1\neq i_1$, $2^{m_{1}}\leq n_1 <2^{m_{1}+1}$ so that $\|T(v_{n_{1}})\|\geq\frac{\varepsilon}{2}$. At the second step we choose $m_{2}>m_{1}$ so that
$$
\max\limits_{2^{m}\leq i<2^{m+1},\,e_i \sqsubseteq e_{i_1}}\|T(v_i)\|\geq\frac{\varepsilon}{2}.
$$
By $(\star)$, we choose $i_{2}$, $2^{m_{2}}\leq i_{2}<2^{m_{2}+1}$ so that $\varepsilon_{i_{2}}=\varepsilon$. Then we choose $m_{2}\neq i_{2}$, $2^{m_{2}}\leq i_{2}<2^{m_{2}+1}$ so that $\|T(v_{m_2})\|\geq\frac{\varepsilon}{2}$. Proceeding further, we construct the desired sequence. Indeed, $\|T(v_{m_i})\|\geq\frac{\varepsilon}{2}$ by the construction and the mutual disjointness for $v_{m_{l}},v_{m_{j}}$, $j\neq l$ is guaranteed by the condition $m_{j}\neq i_{j}$, because the elements $v_{m_{j+l}}$ are fragments of $v_{i_{j}}$ which are disjoint to $v_{m_{j}}$.
\end{proof}

\begin{lemma} \label{le:w88495737}
Let $(V,E)$ be a Banach-Kantorovich space over  atomless Dedekind complete vector lattice $E$ and $F$ be a  finite dimensional Banach space. Then every orthogonally additive laterally-to-norm continuous $C$-compact operator $T: V \to F$ is narrow.
\end{lemma}

\begin{proof}
Fix any $v\in V$ and $\varepsilon>0$. Using Lemma~\ref{le:4.5}, we construct a disjoint tree $(v_{n})$ on $v$ with $\|T(v_{2n})\| = \|T(v_{2n+1})\|$ for all $n\in\Bbb{N}$. By lemma~\ref{le:iiyugte678} we choose $m$ so that $\gamma_m \, {\rm dim} \, F<\varepsilon$. Then using Lemma~\ref{le:rounding}, we choose numbers $\lambda_{i}\in\{0,1\}$ for $i=2^{m},\dots, 2^{m+1}-1$ so that
\begin{equation} \label{eq:855jfg}
\begin{split}
\Bigl\|2\sum_{i=2^{m}}^{2^{m+1}-1} \Bigl(\frac{1}{2}-\lambda_{i} \Bigr) \, T(v_i)\Bigr\|
&\leq {\rm dim} \, F \max\limits_{2^{m}\leq i<2^{m+1}} \|T(v_i)\|\\
&= \gamma_m \, {\rm dim} \, F<\varepsilon.
\end{split}
\end{equation}

Observe that for $I_1 = \{i = 2^m, \ldots, 2^{m+1}-1: \,\, \lambda_i = 0\}$ and $I_2 = \{i = 2^m, \ldots, 2^{m+1}-1: \,\, \lambda_i = 1\}$ the vectors $w_j = \sum_{i \in I_j} v_i$, $j = 1,2$ are MC fragments of $v$ and by \eqref{eq:855jfg},
$$
\|T(w_1) - T(w_2)\| = \Bigl\|\sum\limits_{i=2^{m}}^{2^{m+1}-1} (1-2\lambda_{i}) \, T(v_i)\Bigr\| < \varepsilon.
$$
\end{proof}

The following theorem is the first main result of the section.

\begin{thm} \label{thm:ourmain}
Let $(V,E)$ be a Banach-Kantorovich space over  atomless Dedekind complete vector lattice $E$ and $F$ be a Banach space. Then every orthogonally additive laterally-to-norm continuous $C$-compact operator $T: V \to F$ is narrow.
\end{thm}
\begin{proof}
We may consider $F$ as a subspace of some $l_{\infty}(D)$ space
$$
F \hookrightarrow F^{\star\star} \hookrightarrow l_{\infty}(B_{F^{\star}})=l_{\infty}(D)=W.
$$
By the notation $\hookrightarrow$ we mean an isometric embedding. It is well known that if $H$ is a relatively compact subset of $l_{\infty}(D)$ for some infinite set $D$ and $\varepsilon>0$ then there exists a finite rank operator $S\in l_{\infty}(D)$ such that $\|x-Sx\|\leq\varepsilon$ for every $x\in H$ \cite[Lemma~10.25]{PR}. Fix any $v\in V$ and $\varepsilon>0$. Since $T$ is a $C$-compact operator, $K = \{T(u):\,u \,\,\text{is a fragment of}\,\, w \}$ is relatively compact in $X$ and hence, in $W$. By the above, there exists a finite rank linear operator $S\in\mathcal{L}(W)$ such that $\|w- Sw\|\leq\frac{\varepsilon}{4}$ for every $w\in K$. Then $R=S\circ T$ is an orthogonally additive laterally-norm continuous finite rank operator. By Lemma~\ref{le:w88495737}, there exist $MC$ fragments $v_{1},v_{2}$ of $v$ such that $\|R(v_1) - R(v_2)\| < \frac{\varepsilon}{2}$. Thus,
\begin{align*}
&\|T(v_1) - T(v_2) \| \\
&= \|T(v_1) - T(v_2) + S(T(v_1)) - S(T(v_2)) - S(T(v_1)) + S(T(v_2)) \| \\
&=\|T(v_{1}) - T(v_{2}) + R(v_{1}) - R(v_{2}) - S(T(v_{1})) + S(T(v_{2}))\| \\
&\leq\|R(v_{1}) - R(v_{2})\| + \|T(v_{1}) - S(T(v_{1})) - (T(v_{2}) - S(T(v_{2})))\| \\
&<\frac{\varepsilon}{2}+\frac{\varepsilon}{2}=\varepsilon.
\end{align*}
\end{proof}

Now we  present the second   main result of this section.

\begin{thm} \label{thm:newww}
Let $(V,E)$ be a Banach-Kantorovich space over  atomless Dedekind complete vector lattice $E$ and $\Gamma$ any set. Let $X = X(\Gamma)$ denote one of the Banach lattices $c_0(\Gamma)$ or $\ell_p(\Gamma)$ with $1 \leq p < \infty$. Then every laterally-to-norm continuous dominated Uryson operator $T: V \to X$ is narrow.
\end{thm}
\begin{proof}
Let $T: V \to X$ be a dominated Uryson   operator, $v \in V$ and $\varepsilon > 0$. Now we may write
$$
\ls Tu\rs\leq\ls T\rs\ls u\rs\leq\ls T\rs\ls v\rs
$$
for all $u\sqsubseteq v$. Choose $x \in X^+$ so that $|T(u)| \leq x$ for all $u \sqsubseteq v$. Then we choose a finite subset $\Gamma_0 \subset \Gamma$ so that
\begin{enumerate}
  \item $|x(\gamma)| \leq \varepsilon/4$ for all $\gamma \in \Gamma \setminus \Gamma_0$ if $X = c_0(\Gamma)$ and
  \item $\sum_{\gamma \in \Gamma \setminus \Gamma_0} (x(\gamma) )^p \leq (\varepsilon/4)^p$ if $X = \ell_p(\Gamma)$.
\end{enumerate}
Let $P$ be the projection of $X$ onto $X(\Gamma_0)$ along $X(\Gamma \setminus \Gamma_0)$ and $Q = Id - P$ the orthogonal projection. Obviously, both $P$ and $Q$ are positive linear bounded operators. Since $S = P \circ T: V \to X(\Gamma_0)$ is a finite rank laterally-to-norm continuous operator, by Lemma~\ref{le:w88495737}, $S$ is narrow, and hence, there are MC fragments $v_1$, $v_2$ of $v$ with $\|S(v_1) - S(v_2)\| < \varepsilon/2$. Since $|T(v_i)| \leq x$, by the positivity of $Q$ we have that $Q(Tv_i) \leq Qx$, and thus, $\|Q(T(v_i))\| \leq \|Q(x)\|$ for $i = 1,2$. Moreover, by (1) and (2), $\|Q(x)\| \leq \varepsilon/4$. Then
\begin{align*}
\|T(v_1) - T(v_2)\| &= \|S(v_1) + Q(T(v_1)) - S(v_2) - Q(T(v_2))\| \\
&\leq \|S(v_1) - S(v_2)\| + \|Q(T(v_1))\| + \|Q(T(v_2))\| \\
&< \, \frac{\varepsilon}{2} \, + \|Q(x)\| + \|Q(x)\| \leq \varepsilon.
\end{align*}
\end{proof}

The idea used in the proof of Theorem~\ref{thm:newww} could be generalized as follows.

\begin{definition}
Let  $F$ be ordered vector space. We say that a linear operator $G: F \to F$ is \textit{quasi-monotone with a constant} $M > 0$ if for each $x,y \in F^+$ the inequality $x \leq y$ implies $Gx \leq M Gy$. An operator $G: F \to F$ is said to be \textit{quasi-monotone} if it is quasi-monotone with some constant $M > 0$.
\end{definition}

If $G \neq 0$ in the above definition, we easily obtain $M \geq 1$. Observe also that the quasi-monotone operators with constant $M = 1$ exactly are the positive operators.

Recall that a sequence of elements $(f_n)_{n=1}^\infty$ (resp., of finite dimensional subspaces $(F_n)_{n=1}^\infty$) of a Banach space $F$ is called a \textit{basis} (resp., a \textit{finite dimensional decomposition}, or \textit{FDD}, in short) if for every $f \in F$ there exists a unique sequence of scalars $(a_n)_{n=1}^\infty$ (resp. sequence $(u_n)_{n=1}^\infty$ of elements $u_n \in F_n$) such that $f = \sum_{n=1}^\infty a_n f_n$ (resp., $e = \sum_{n=1}^\infty u_n$). Every basis $(f_n)$ generates the FDD $F_n = \{\lambda f_n: \, \lambda \in \mathbb R\}$. Any basis $(f_n)$ (resp., any FDD $(F_n)$) of a Banach space generates the corresponding \textit{basis projections} $(P_n)$ defined by
$$
P_n \Bigl( \sum_{k=1}^\infty a_k e_k \Bigr) = \sum_{k=1}^n a_k f_k \,\,\,\,\, \left( \mbox{resp.,} \,\, P_n \Bigl( \sum_{k=1}^\infty u_k \Bigr) = \sum_{k=1}^n u_k \right),
$$
which are uniformly bounded. For more details about these notions we refer the reader to \cite{Li-1}. The orthogonal projections to $P_n$'s defined by $Q_n = Id - P_n$, where $Id$ is the identity operator on $F$, we will call the \textit{residual projections} associated with the basis $(f_n)_{n=1}^\infty$ (resp., to the FDD $(F_n)_{n=1}^\infty$).

\begin{definition}
A basis $(f_n)$ (resp., an FDD $(F_n)$) of a Banach lattice $F$ is called \textit{residually quasi-monotone} if there is a constant $M > 0$ such that the corresponding residual projections are quasi-monotone with constant $M$.
\end{definition}

In other words, an FDD $(F_n)$ of $F$ is residually quasi-monotone if the corresponding approximation of smaller in modulus elements is better, up to some constant multiple: if $x,y \in F$ with $|x| \leq |y|$ then $\|x - P_nx\| \leq M \|y - P_ny\|$ for all $n$ (observe that $\|z - P_nz\| \to 0$ as $n \to \infty$ for all $z \in F$).
 
\begin{thm} \label{thm:newww2}
Let $(V,E)$ be a Banach-Kantorovich space over  atomless Dedekind complete vector lattice $E$ and $F$ be  a Banach lattice with a residually quasi-monotone basis or, more general, a residually quasi-monotone FDD. Then every dominated  Uryson operator $T: V \to F$ is narrow.
\end{thm}
\begin{proof}
Let $(F_n)$ be an FDD of $F$ with the corresponding projections $(P_n)$, and let $M > 0$ be such that for every $n \in \mathbb N$ the operator $Q_n = Id - P_n$ is quasi-monotone with constant $M$. Let $T: V \to F$ be a dominated Uryson operator, $v \in V$ and $\varepsilon > 0$. Choose $f \in F_+$ so that $|Tu| \leq f$ for all $u \sqsubseteq v$. Since $\lim_{n \to \infty} P_n f = f$, we have that $\lim_{n \to \infty} Q_n f = 0$. Choose $n$ so that
\begin{equation} \label{eq:djjf7}
\|Q_n f\| \leq \frac{\varepsilon}{4M} \, .
\end{equation}
Since $S = P_n \circ T: V \to E_n$ is a finite rank dominated  Uryson operator, by Lemma~\ref{le:w88495737}, $S$ is narrow, and hence, there are MC fragments $v_1$, $v_2$ of $v$ such that $\|Sv_1 - Sv_2\| < \varepsilon/2$. Since $|Tv_i| \leq f$, by the quasi-monotonicity of $Q_n$ we have that $\|Q_n(Tv_i)\| \leq M \|Q_n f\|$ for $i = 1,2$. Then by \eqref{eq:djjf7},
\begin{align*}
\|Tv_1 - Tv_2\| &= \|Sv_1 + Q(Tv_1) - Sv_2 - Q(Tv_2)\| \\
&\leq \|Sv_1 - Sv_2\| + \|Q ( T v_1)\| + \|Q ( T v_2)\| \\
&< \, \frac{\varepsilon}{2} \, + M\|Qf\| + M\|Qf\| \leq \varepsilon.
\end{align*}
\end{proof}

\section{Domination problem for narrow operators}
\label{sec5}

In this section we consider a domination problem for the exact dominant of   dominated  Uryson operators.
Firstly, define  the following set
$$
\widetilde{E}_{+}=\{e\in E_{+}:\,e=\bigsqcup\limits_{i=1}^{n}\ls v_{i}\rs;\,v_{i}\in V;\,n\in\Bbb{N}\}.
$$
\begin{thm}(\cite{PP-1},Theor.~3.4., 3.7.)\label{dom-01}
Let $(V,E)$, $(W,F)$ be lattice-normed spaces,   with $V$ decomposable and $F$ Dedekind complete. Then every dominated Uryson operator $T:V\to W$ has an exact dominant $\ls T\rs$ and it can be calculated by the following formulas:
\begin{enumerate}
  \item[$(1)$] $\ls T\rs(e)=\sup\Big\{\sum\limits_{i=1}^{n}\ls Tu_{i}\rs:\,\coprod\limits_{i=1}^{n}\ls u_{i}\rs=e,\, n\in\Bbb{N}\Big\}$\,$(e\in\widetilde{E}_{+})$;
  \item[$(2)$] $\ls T\rs(e)=\sup\Big\{\ls T\rs(e_{0}):\,e_{0}\in\widetilde{E}_{+},\, e_{0}\sqsubseteq e\Big\}$;\,$(e\in E_{+})$
  \item[$(3)$] $\ls T\rs(e)=\ls T\rs(e_{+})+\ls T\rs(e_{-})$,\,$e\in E$.
\end{enumerate}
\end{thm}

\begin{thm}\label{th:narmod}
Let $(V,E)$ be a lattice-normed space with $E$ atomless, $(W,F)$ be  a Banach-Kantorovich space,  $F$ be Dedekind  complete and   $T$ be an order narrow dominated Uryson operator from $V$ to $W$. Then the exact dominant $\ls T\rs:E\to F$ is order narrow.
\end{thm}

\begin{proof}
Firstly, fix any $e\in\widetilde{E}_{+}$ and $\varepsilon>0$.  Since
$$
\left\{\sum\limits_{i=1}^n\ls Tv_i\rs:
\sum\limits_{i=1}^n\ls v_i\rs= e;\,\ls v_{i}\rs\bot\ls v_{j}\rs;\,i\neq j;\,n\in\Bbb{N}\right\}
$$
is an increasing net,  there exits  a net of finite collections $\{v^{\alpha}_{1},\dots,v^{\alpha}_{n_{\alpha}}\}\subset V$, $\alpha\in\Lambda$ with
$$
e=\bigsqcup\limits_{i=1}^{n_{\alpha}}\ls v_{i}^{\alpha}\rs,\,\alpha\in\Lambda
$$
and
$$
\Big(\ls T\rs(e)-\sum\limits_{i=1}^{n_{\alpha}}\ls Tv_{i}^{\alpha}\rs\,\Big)\leq y_{\alpha}\overset{(o)}\longrightarrow 0,
$$
where $0\leq y_{\alpha}$, $\alpha\in\Lambda$ is an decreasing net and $\inf(y_{\alpha})_{\alpha\in\Lambda}=0$. Fix some $\alpha\in\Lambda$. Since $T$ is order narrow operator we may assume that there exist a finite set of a nets of a decompositions $v_{i}^{\alpha}=u_{i}^{\beta_{\alpha}}\sqcup w_{i}^{\beta_{\alpha}}$, $i\in\{1,\dots,n_{\alpha}\}$ which depends of $\alpha$, indexed by the same set $\Delta$, such that $\ls Tu_{i}^{\beta_{\alpha}}-Tw_{i}^{\beta_{\alpha}}\rs\overset{(
o)}\longrightarrow 0$, $i\in\{1,\dots, n_{\alpha}\}$. Let $f^{\beta_{\alpha}}=\coprod\limits_{i=1}^{n_{\alpha}}\ls u_{i}^{\beta_{\alpha}}\rs$ and $g^{\beta_{\alpha}}=\coprod\limits_{i=1}^{n_{\alpha}}\ls w_{i}^{\beta_{\alpha}}\rs$. Then we have
\begin{align*}
0\leq\ls T\rs(f^{\beta_{\alpha}})-\sum\limits_{i=1}^{n_{\alpha}}(\ls Tu_{i}^{\beta_{\alpha}}\rs)\leq\\
\ls T\rs(e)-\sum\limits_{i=1}^{n_{\alpha}}\ls Tv_{i}^{\alpha}\rs;\\
0\leq\ls T\rs(g^{\beta_{\alpha}})-\sum\limits_{i=1}^{n_{\alpha}}(\ls Tw_{i}^{\beta_{\alpha}}\rs)\leq\\
\ls T\rs(e)-\sum\limits_{i=1}^{n_{\alpha}}\ls Tv_{i}^{\alpha}\rs.
\end{align*}
Now we may write
\begin{align*}
\Big|\ls T\rs f^{\beta_{\alpha}}-\ls T\rs g^{\beta_{\alpha}}\Big|=\\
\Big|\ls T\rs f^{\beta_{\alpha}}-\sum\limits_{i=1}^{n_{\alpha}}\ls Tu_{i}^{\beta_{\alpha}}\rs+
\sum\limits_{i=1}^{n_{\alpha}}\ls Tu_{i}^{\beta_{\alpha}}\rs-
\sum\limits_{i=1}^{n_{\alpha}}\ls Tw_{i}^{\beta_{\alpha}}\rs+
\sum\limits_{i=1}^{n_{\alpha}}\ls Tw_{i}^{\beta_{\alpha}}\rs-\ls T\rs g^{\beta_{\alpha}}\Big|\leq\\
\Big|\ls T\rs f^{{\beta_{\alpha}}}-\sum\limits_{i=1}^{n_{\beta_{\alpha}}}\ls Tu_{i}^{\beta_{\alpha}}\rs\Big|+
\Big|\ls T\rs g^{\beta_{\alpha}}-\sum\limits_{i=1}^{n_{\alpha}}\ls Tw_{i}^{\alpha}\rs\Big|+
\Big|\sum\limits_{i=1}^{n_{\alpha}}\ls Tu_{i}^{\beta_{\alpha}}\rs-\sum\limits_{i=1}^{n_{\alpha}}\ls Tw_{i}^{\beta_{\alpha}}\rs\Big|\leq\\
2\Big(\ls T\rs(e)-\sum\limits_{i=1}^{n_{\alpha}}\ls Tv_{i}^{\alpha}\rs\Big)+
\sum\limits_{i=1}^{n_{\alpha}}\Big|\ls Tu_{i}^{\beta_{\alpha}}\rs-\ls Tw_{i}^{\beta_{\alpha}}\rs\Big|\leq\\
\Big(2\ls T\rs(e)-2\sum\limits_{i=1}^{n_{\alpha}}\ls Tv_{i}^{\alpha}\rs+\sum\limits_{i=1}^{n_{\alpha}}\ls Tu_{i}^{\beta_{\alpha}}- Tw_{i}^{\beta_{\alpha}}\rs\Big)\overset{(o)}\longrightarrow 0.
\end{align*}
Therefore $e=f^{\beta_{\alpha}}\sqcup g^{\beta_{\alpha}}$, $\alpha\in\Lambda$, $\beta_{\alpha}\in\Delta$ is a desirable net of decompositions. Now, let $e\in E_{+}$. Observe that
$D=\{f\sqsubseteq e:\,f\in\widetilde{E}_{+}\}$
is a directed set, where $f_{1}\leq f_{2}$ mean that $f_{1}\sqsubseteq f_{2}$.
Indeed, let $f_{1}=\coprod\limits_{i=1}^{k}\ls u_{i}\rs, f_{1}\sqsubseteq e$,  $f_{2}=\coprod\limits_{j=1}^{n}\ls w_{j}\rs; f_{2}\sqsubseteq e$, $u_{i}, w_{j}\in V$, $1\leq i\leq k$, $1\leq j\leq n$. Then by decomposability of the vector norm there exists a set  of mutually disjoint elements $(v_{ij})$, $1\leq i\leq k$, $1\leq j\leq n$, such that
$u_{i}=\coprod\limits_{j=1}^{n}v_{ij}$ for every $1\leq i\leq k$ and $w_{j}=\coprod\limits_{i=1}^{k}v_{ij}$ for every $1\leq j\leq n$. Let $f=\coprod\ls v_{ij}\rs$. It is clear that $\ls T\rs f_{i}\leq\ls T\rs f$, $i\in\{1,2\}$. Let $(e_{\alpha})_{\alpha\in\Lambda}, e_{\alpha}\in D$ be a net, where $\ls T\rs=\sup\limits_{\alpha}\ls T\rs e_{\alpha}$. Fix $\alpha\in\Lambda$, then for $e_{\alpha}\in D$ there exists a net of decompositions $e_{\alpha}=f_{\alpha}^{\beta}\sqcup g_{\alpha}^{\beta}$, $\beta\in\Delta$, such that $\Big|\ls T\rs f_{\alpha}^{\beta}-\ls T\rs g_{\alpha}^{\beta}\Big|\overset{(o)}\longrightarrow 0$. Thus we have
\begin{align*}
\Big| \ls T\rs(e-e_{\alpha}+f_{\alpha}^{\beta})-\ls T\rs g_{\alpha}^{\beta}\Big|=\\
\Big| \ls T\rs(e-e_{\alpha})+\ls T\rs f_{\alpha}^{\beta}-\ls T\rs g_{\alpha}^{\beta}\Big|\leq\\
 \Big(\ls T\rs e-\ls T\rs e_{\alpha}+
 \Big|\ls T\rs f_{\alpha}^{\beta}-\ls T\rs g_{\alpha}^{\beta}\Big|\Big)\overset{(o)}\longrightarrow 0.
\end{align*}
Hence $e=((e-e_{\alpha})\sqcup f_{\alpha}^{\beta}))\sqcup g_{\alpha}^{\beta}$ is a desirable net of decompositions.
Since $\ls T\rs\in\mathcal{U}_{+}^{ev}(E,F)$, then $\ls T\rs(e)=\ls T\rs(-e)$ for every $e\in(-E_{+})$ and if $e=f_{\alpha}\sqcup g_{\alpha}$ is a necessary net of decompositions for $e$, then $-e=(-f_{\alpha})\sqcup (-g_{\alpha})$ is a same. Finally for arbitrary element $e\in E$ we have $e=e_{+}-e_{-}$ and by (\ref{dom-01}.3) we have $\ls T\rs(e)=\ls T\rs(e_{+})+\ls T\rs(e_{-})$. Thus, if $e_{+}=f_{1}^{\alpha}\sqcup f_{2}^{\alpha}$ and $e_{-}=g_{1}^{\alpha}\sqcup g_{2}^{\alpha}$ are necessary nets of decompositions, then
\begin{align*}
\Big|\ls T\rs(f_{1}^{\alpha}+g_{1}^{\alpha})-\ls T\rs(f_{2}^{\alpha}+g_{2}^{\alpha})\Big|=\\
\Big|\ls T\rs(f_{1}^{\alpha}-\ls T\rs(f_{2}^{\alpha}+\ls T\rs g_{1}^{\alpha})-\ls T\rs g_{2}^{\alpha})\Big|\leq\\
\Big(\Big|\ls T\rs f_{1}^{\alpha}-\ls T\rs f_{2}^{\alpha}\Big|+\Big|\ls T\rs g_{1}^{\alpha})-\ls T\rs g_{2}^{\alpha})\Big|\Big)\overset{(o)}\longrightarrow 0
\end{align*}
and $e=(f_{1}^{\alpha}+g_{1}^{\alpha})\sqcup(f_{2}^{\alpha}+g_{2}^{\alpha})$ is a desirable net of decompositions.
\end{proof}
\begin{rem}
This  is an open question. Does the order narrowness of the operator $\ls T\rs$   implies the order narrowness of the  $T$?
A particular case was proved  in (\cite{PP}, Theorem~4.1).
\end{rem}

\end{document}